\documentclass[11pt]{article}
\usepackage{amsmath,amsfonts,latexsym, xspace,amsthm,url,graphicx}
\usepackage{enumerate}
\usepackage[letterpaper,left=1in,right=1in,top=1in,bottom=1in]{geometry}

\usepackage{mathtools}
\usepackage{bbm}
\mathtoolsset{showonlyrefs}
\usepackage{color}

\newcommand{\E}[1]{\textbf{E} \left[#1\right]}

\newcommand{\set}[1]{\left\{#1\right\}}

\newcommand{\lfrac}[2]{\left(\frac{#1}{#2}\right)}

\def\cD{{\cal D}}
\def\cE{{\cal E}}

\def\cF{{\cal F}}

\def\hw{\hat{w}}

\def\A{{\cal A}}

\def\cF{{\cal F}}

\def\a{\alpha}
\def\b{\beta}
\def\d{\delta}

\def\e{\varepsilon}
\def\f{\phi}
\def\g{\gamma}
\def\G{\Gamma}

\def\z{\zeta}

\def\l{\lambda}
\def\m{\mu}
\def\n{\nu}
\def\p{\pi}

\def\r{\rho}

\def\s{\sigma}

\def\t{\tau}
\def\om{\omega}

\def\whp{{\bf whp}}
\newcommand\Prob[1]{{\mbox{Pr}\left\{#1\right\}}}

\def\Pr{{\bf Pr}}

\newtheorem{lemma}{Lemma}
\newtheorem{theorem}{Theorem}
\newtheorem{corollary}{Corollary}

\newcommand{\brac}[1]{\left( #1\right)}
\newcommand{\bfrac}[2]{\brac{\frac{#1}{#2}}}
\newtheorem{conjecture}{Conjecture}

\newcommand{\bfm}[1]{\mbox{\boldmath $#1$}}
\newcommand{\reals}{\mbox{\bfm{R}}}

\newcommand{\rdup}[1]{\lceil #1 \rceil}

\newcommand{\proofend}{\hspace*{\fill}\mbox{$\Box$}}

\newcommand{\beq}[2]{\begin{equation}\label{#1}#2\end{equation}}

\def\vG{\vec{\Gamma}}
\def\vX{\vec{X}}

\def\hw{\widehat{w}}
\def\hC{\widehat{C}}
\def\hM{\widehat{M}}
\def\A{{\boldsymbol \a}}
\title{Minimum-cost matching in a random graph with random costs}
\author{Alan Frieze\thanks{Research supported in part by NSF Grant  DMS1362785. Email: alan@random.math.cmu.edu}\ \ \  Tony Johansson\thanks{Research supported in part by NSF Grant  DMS1362785. Email: tjohanss@andrew.cmu.edu}\\ Department of Mathematical Sciences\\Carnegie Mellon University\\Pittsburgh PA 15213\\U.S.A.}

\begin{document}
\maketitle
\begin{abstract}
Let $G_{n,p}$ be the standard Erd\H{o}s-R\'enyi-Gilbert random graph and let $G_{n,n,p}$ be the random bipartite graph on $n+n$ vertices, where each $e\in [n]^2$ appears as an edge independently with probability $p$. For a graph $G=(V,E)$, suppose that each edge $e\in E$ is given an independent uniform exponential rate one cost. Let $C(G)$ denote the random variable equal to the length of the minimum cost perfect matching, assuming that $G$ contains at least one. We show that w.h.p. if $d=np\gg(\log n)^2$ then w.h.p. $\E{C(G_{n,n,p})} =(1+o(1))\frac{\p^2}{6p}$. This generalises the well-known result for the case $G=K_{n,n}$. We also show that w.h.p. $\E{C(G_{n,p})} =(1+o(1))\frac{\p^2}{12p}$ along with concentration results for both types of random graph. 
\end{abstract}
\section{Introduction}
There are many results concerning the optimal value of combinatorial optimization problems with random costs. Sometimes the costs are associated with $n$ points generated uniformly at random in the unit square $[0,1]^2$. In which case the most celebrated result is due to Beardwood, Halton and Hammersley \cite{BHH} who showed that the minimum length of a tour through the points a.s. grew as $\b n^{1/2}$ for some still unknown $\b$. For more on this and related topics see Steele \cite{S1}.

The optimisation problem in \cite{BHH} is defined by the distances between the points. So, it is defined by a random matrix where the entries are highly correlated. There have been many examples considered where the matrix of costs contains independent entries. Aside from the Travelling Salesperson Problem, the most studied problems in combinatorial optimization are perhaps, the shortest path problem; the minimum spanning tree problem and the matching problem. As a first example, consider the shortest path problem in the complete graph $K_n$ where the edge lengths are independent exponential random variables with rate 1. We denote the exponential random variable with rate $\l$ by $E(\l)$. Thus $\Pr(E(\l)\geq x)=e^{-\l x}$ for $x\in \reals$. Janson \cite{J99} proved (among other things) that if $X_{i,j}$ denotes the shortest distance between vertices $i,j$ in this model then $\E{X_{1,2}}=\frac{H_n}{n}$ where $H_n=\sum_{i=1}^n\frac{1}{i}$. 

As far as the spanning tree problem is concerned, the first relevant result is due to Frieze \cite{F85}. He showed that if the edges of the complete graph are given independent uniform $[0,1]$ edge weights, then the (random) minimum length of a spanning tree $L_n$ satisfies $\E{L_n}\to\z(3)=\sum_{i=1}^\infty\frac{1}{i^3}$ as $n\to\infty$. Further results on this question can be found in Steele \cite{Steele}, Janson \cite{J95}, Beveridge, Frieze and McDiarmid \cite{BFM98}, Frieze, Ruszinko and Thoma \cite{FRT00} and Cooper, Frieze, Ince, Janson and Spencer \cite{CFIJS}.

In the case of matchings, the nicest results concern the the minimum cost of a matching in a randomly edge-weighted copy of the complete bipartite graph $K_{n,n}$. If $C_n$ denotes the (random) minimum cost of a perfect matching when edges are given independent exponential $E(1)$ random variables then the story begins with Walkup \cite{W80a} who proved that $\E{C_n}\leq 3$. Later Karp \cite{K87} proved that $\E{C_n}\leq 2$. Aldous \cite{A92}, \cite{A01} proved that $\lim_{n\to\infty}\E{C_n}=\z(2)=\sum_{k=1}^\infty\frac{1}{k^2}$.
 Parisi \cite{P98} conjectured that in fact $\E{C_n}= \sum_{k=1}^{n}\frac{1}{k^2}$. This was proved independently by Linusson and W\"astlund \cite{LW04}
and by  Nair, Prabhakar and Sharma \cite{NPS05}. A short elegant proof was given by W\"astlund \cite{W1}, \cite{W2}.

In the paper \cite{BFM98} on the minimum spanning tree problem, the complete graph was replaced by a $d$-regular graph $G$. Under some mild expansion assumptions, it was shown that if $d\to\infty$ then $\z(3)$ can be replaced asymptotically by $n\z(3)/d$. 

Now consider a $d$-regular bipartite graph $G$ on $2N$ vertices. Here $d = d(N) \rightarrow \infty$ as $N \rightarrow \infty$. Each edge $e$ is assigned a cost $w(e)$, each independently chosen according to the exponential distribution $E(1)$. Denote the total cost of the minimum-cost perfect matching by $C(G)$.

We conjecture the following (under some possibly mild restrictions):
\begin{conjecture}\label{conj1}
Suppose $d = d(N) \rightarrow \infty$ as $N \rightarrow \infty$. For any $d$-regular bipartite $G$,
$$\E{C(G)} =(1+o(1))\frac{N}{d} \frac{\p^2}{6}.$$
Here the $o(1)$ term goes to zero as $N\to\infty$.
\end{conjecture}
In this paper we prove the conjecture for random graphs and random bipartite graphs. Let $G_{n,n,p}$ be the random bipartite graph on $n+n$ vertices, where each $e\in [n]^2$ appears as an edge independently with probability $p$. Suppose that each edge $e$ is given an independent uniform exponential rate one cost. 
\begin{theorem}\label{th1}
If $d=np=\om(\log n)^2$ where $\om\to\infty$ then w.h.p. $\E{C(G_{n, n, p})}\approx\frac{\p^2}{6p}$.
\end{theorem}
Here the statement is that for almost all graphs $G=G_{n,n,p}$ we have $\E{C(G)}\approx \frac{\p^2}{6p}$. We will in fact show that in this case $C(G)$ will be highly concentrated around $\frac{\p^2}{6p}$.

Here $A_n\approx B_n$ iff $A_n=(1+o(1))B_n$ as $n\to \infty$ and the event $\cE_n$ occurs with high probability (w.h.p.) if $\Pr(\cE_n)=1-o(1)$ as $n\to \infty$. 

In the case of $G_{n, p}$ we prove
\begin{theorem}\label{th2}
If $d = np = \omega(\log n)^2$ where $\om\to\infty$ then w.h.p. $\E{C(G_{n, p})} \approx\frac{\pi^2}{12p}$.
\end{theorem}

Applying results of Talagrand \cite{tal} we can prove the following concentration result.
\begin{theorem}\label{thtal}
Let $\e>0$ be fixed, then
$$
\Pr\brac{\left|C(G_{n, n, p})-\frac{\p^2}{6p}\right|\geq \frac{\e}{p}}\leq n^{-K}, \quad \Pr\brac{\left|C(G_{n, p})-\frac{\p^2}{12p}\right|\geq \frac{\e}{p}}\leq n^{-K}
$$
for any constant $K>0$ and $n$ large enough.
\end{theorem}
\section{Proof of Theorem \ref{th1}}\label{pot1}
We find that the proofs in \cite{W1}, \cite{W2} can be adapted to our current situation. Suppose that the vertices of $ G=G_{n,n,p}$ are denoted $A=\set{a_i,i\in [n]}$ and $B=\set{b_j,j\in [n]}$. Let $C(n,r)$ denote the cost of the minimum cost matching 
$$M_r=\set{(a_i,\f_r(a_i)):i=1,2,\ldots,r}\text{ of }A_r=\set{a_1,a_2,\ldots,a_r}\text{ into }B.$$ 
We will prove that w.h.p. 
\beq{eq1}{
\E{C(n,r)-C(n,r-1)}\approx \frac1{p}\sum_{i=0}^{r-1}\frac{1}{r(n-i)}.
}
for $r=1,2,\ldots,n-m$ where
$$m=\bfrac{n}{\om^{1/2}\log n}.$$

Using this we argue that w.h.p.
\beq{eq2}{
\E{C(G)}=\E{C(n,n)}=\E{C(n,n)-C(n,n-m+1)}+\frac{1+o(1)}{p} \sum_{r=1}^{n-m}\sum_{i=0}^{r-1}\frac{1}{r(n-i)}.
}
We will then show that 
\begin{align}
\sum_{r=1}^{n-m}\sum_{i=0}^{r-1}\frac{1}{r(n-i)}&\approx\sum_{k=1}^\infty \frac{1}{k^2}=\frac{\p^2}{6}.\label{sum=}\\
\E{C(n,n)-C(n,n-m+1)}&=o(p^{-1})\quad w.h.p.\label{Cnm=}
\end{align}
Theorem \ref{th1} follows from these two statements.
\subsection{Outline of the proof}
We first argue (Lemma \ref{lem1}) that $B_r = \f(A_r)$ is a uniformly random set. This enables us to show (Lemma \ref{lem2}) that w.h.p. vertices $v\in A_r$ have aproximately $(n-r)p$ neighbors in $B\setminus B_r$.  
Then comes the beautiful idea of adding a vertex $b_{n+1}$ and joining it to every vertex in $A$ by an edge of cost $E(\l)$. The heart of the proof is in Lemma \ref{lem10a} that relates $\E{C(n, r) - C(n, r-1)}$ in a precise way to the probability that $b_{n+1}$ is covered by $M_r^*$, the minimum cost matching of $A_r$ into $B^*=B\cup\set{b_{n+1}}$. The proof now focuses on estimating this probability $P(n,r)$. If $r$ is not too close to $n$ then this probability can be estimated (see \eqref{innot}) by careful conditioning and the use of properties of the exponential distribution. From thereon, it is a matter of analysing the consequences of the estimate for $\E{C(n, r) - C(n, r-1)}$ in \eqref{13}. The final part of the proof involves showing that $\E{C(n,n)-C(n-m+1)}$ is insignificant. This essentially boils down to showing that w.h.p. no edge in the minimum cost matching has cost more than $O(\log n/(np))$. 

\subsection{Proof details}
Let $B_r=\set{\f_r(a_i):i=1,2,\ldots,r}$.
\begin{lemma}\label{lem1}
$B_r$ is a random $r$-subset of $B$.
\end{lemma}
\begin{proof}
Let $L$ denote the $n\times n$ matrix of edge costs, where $L(i,j)=W(a_i,b_j)$ and $L(i,j)=\infty$ if edge $(a_i,b_j)$ does not exist in $G$. For a permutation $\p$ of $B$ let $L_\p$ be defined by $L_\p(i,j)=L(i,\p(j))$. Let $X,Y$ be two distinct $r$-subsets of $B$ and let $\p$ be any permutation of $B$ that takes $X$ into $Y$. Then we have
$$\Pr(B_r(L)=X)=\Pr(B_r(L_\p)=\p(X))=\Pr(B_r(L_\p)=Y)=\Pr(B_r(L)=Y),$$
where the last equality follows from the fact that $L$ and $L_\p$ have the same distribution.
\end{proof}

We use the above lemma and the Chernoff bounds to bound degrees. For reference we use the following: Let $B(n,p)$ denote the binomial random variable with parameters $n,p$. Then for $0\leq\e\leq 1$ and $\a>0$,
\begin{align}
\Pr(B(n,p)\leq (1-\e)np)&\leq e^{-\e^2np/2}.\label{chern1}\\
\Pr(B(n,p)\leq (1+\e)np)&\leq e^{-\e^2np/3}.\label{chern2}\\
\Pr(B(n,p)\geq \a np)&\leq\bfrac{e}{\a}^{\a np}.\label{chern3}
\end{align}
For $v\in A$ let $d_r(v)=|\set{w\in B\setminus B_r:(v,w)\in E(G)}|$. Then we have the following lemma:
\begin{lemma}\label{lem2}
$$|d_r(v)-(n-r)p|\leq \om^{-1/5}(n-r)p\text{ w.h.p. for }v\in A,\,0\leq r\leq n-m.$$
\end{lemma}
\begin{proof}
This follows from Lemma \ref{lem1} i.e. $B\setminus B_r$ is a random set and the Chernoff bounds \eqref{chern1}, \eqref{chern2} with $\e=\om^{-1/5}$ viz.
\begin{align*}
\Pr(\exists v:|d_r(v)-(n-r)p|\geq \om^{-1/5}(n-r)p)&\leq 2ne^{-\om^{-2/5}(n-r)p/3}\\
&\leq 2n^{1-\om^{1/10}/3}.
\end{align*}
\end{proof}

We can now use the ideas of \cite{W1}, \cite{W2}. We add a special vertex $b_{n+1}$ to $B$, with edges to all $n$ vertices of $A$. Each edge adjacent to $b_{n+1}$ is assigned an $E(\l)$ cost independently, $\l>0$. We now consider $M_r$ to be a minimum cost matching of $A_r$ into $B^*=B\cup \set{b_{n+1}}$. We denote this matching by $M_r^*$ and we let $B_r^*$ denote the corresponding set of vertices of $B^*$ that are covered by $M_r^*$.

Define $P(n, r)$ as the normalized probability that $v_{n+1}$ participates in $M_r^*$, i.e.
\begin{equation}
 P(n, r) = \lim_{\l \rightarrow 0} \frac{1}{\l}\Prob{b_{n+1}\in B_r^*}.
\end{equation}
Its importance lies in the following lemma:
\begin{lemma}\label{lem10a}
\begin{equation}
\E{C(n, r) - C(n, r-1)} = \frac{1}{r}P(n, r).
\end{equation}
\end{lemma}
\begin{proof}
Let $X=C(n,r)$ and let $Y=C(n,r-1)$. Fix $i\in [r]$ and let $w$ be the cost of the edge $(a_i, b_{n+1})$, and let $I$ denote the indicator variable for the event that the cost of the cheapest $A_{r}$-assignment that contains this edge is smaller than the cost of the cheapest $A_{r}$-assignment that does not use $b_{n+1}$. In other words, $I$ is the indicator variable for the event $\{Y + w < X\}$.

If $(a_i, b_{n+1}) \in M_r^*$ then $w < X - Y$. Conversely, if $w < X - Y$ and no other edge from $b_{n+1}$ has cost smaller than $X - Y$, then $(a_i, b_{n+1})\in M_r^*$, and when $\l \to 0$, the probability that there are two distinct edges from $b_{n+1}$ of cost smaller than $X - Y$ is of order $O(\l^2)$.

Since $w$ is $E(\l)$ distributed, as $\l\to 0$ we have,
\begin{equation}
\E{X - Y} = \frac{d}{d\l} \E{I}\bigg |_{\l = 0} = \lim_{\l \to 0} \frac{1}{\l} \Prob{w < X - Y} = \frac{1}{r}P(n, r).
\end{equation}
The factor $1/r$ comes from each $i\in[r]$ being equally likely to be incident to the matching edge containing $b_{n+1}$, if it exists.
\end{proof}
We now proceed to estimate $P(n,r)$.
\begin{lemma}\label{lem3}
Suppose $r < n-m$. Then 
\begin{equation}\label{innot}
\Pr(b_{n+1}\in B_{r}^*\mid b_{n+1}\notin B_{r-1}^*)=\frac{\l}{p(n-r+1)(1 + \e_r) + \l}
\end{equation}
where $|\e_r| \leq \om^{-1/5}$.
\end{lemma}
\begin{proof}
Assume that $b_{n+1}\notin B_{r-1}^*$. $M_{r}^*$ is obtained from $M_{r-1}^*$ by finding an augmenting path $P=(a_{r},\ldots,a_\s,b_\t)$ from $a_{r}$ to $B^*\setminus B_{r-1}^*$ of minimum additional cost. Let $\a=W(\s,\t)$. We condition on (i) $\s$, (ii) the lengths of all edges other than $(a_\s,b_j),b_j\in B^*\setminus B_{r-1}^*$ and (iii) $\min\set{W(\s,j):b_j\in B^*\setminus B_{r-1}^*}=\a$. With this conditioning $M_{r-1}=M_{r-1}^*$ will be fixed and so will $P'=(a_{r},\ldots,a_\s)$. We can now use the following fact: Let $X_1,X_2,\ldots,X_M$ be independent exponential random variables of rates $\a_1,\a_2,\ldots,\a_M$. Then the probability that $X_i$ is the smallest of them is $\a_i/(\a_1+\a_2+\cdots+\a_M)$. Furthermore, the probability stays the same if we condition on the value of $\min\set{X_1,X_2,\ldots,X_M}$. Thus
$$\Pr(b_{n+1}\in B_{r}^*\mid b_{n+1}\notin B_{r-1}^*)=\frac{\l}{d_{r-1}(a_\s) + \l}.$$
\end{proof}

\begin{corollary}\label{corollary1}
\begin{equation}\label{13}
P(n,r)=\frac{1}{p}\left(\frac{1}{n} + \frac{1}{n-1} + \dots + \frac{1}{n-r+1}\right)(1+\e_r)
\end{equation}
where $|\e_r| \leq \om^{-1/5}$.
\end{corollary}
\begin{proof}
Let $\n(j) = p(n-j)(1 + \e_j)$, $|\e_j| \leq \om^{-1/5}$. Then 
\begin{align*}
\Pr(b_{n+1}\in B_r^*)&= 1 - \frac{\n(0)}{\n(0) + \l} \cdot \frac{\n(1)}{\n(1) + \l} \cdots \frac{\n(r-1)}{\n(r-1)+\l} \\
&=1 - \left(1 + \frac{\l}{\n(0)}\right)^{-1}\cdots \left(1 + \frac{\l}{\n(r-1)}\right)^{-1} \\
&= \left(\frac{1}{\n(0)} + \frac{1}{\n(1)} + \dots + \frac{1}{\n(r-1)}\right)\l + O(\l^2) \\
&=\frac{1}{p}\left(\frac{1}{n(1 + \e_0)} + \frac{1}{(n-1)(1 + \e_1)} + \dots + \frac{1}{(n-r+1)(1 + \e_{r-1})}\right) \l + O(\l^2)
\end{align*}
and each error factor satisfies $|1 - 1 / (1 + \e_j)| \leq \om^{-1/5}$. Letting $\l\to 0$ gives the lemma.
\end{proof}

\begin{lemma}\label{lem4}
If $r \leq n-m$ then
\begin{equation}
\E{C(n,r) - C(n,r-1)} =\frac{1+o(1)}{rp} \sum_{i = 0}^{r-1} \frac{1}{n-i}
\end{equation}
where $|\e_k| \leq \om^{-1/5}$.
\end{lemma}
\begin{proof}
This follows from Lemma \ref{lem10a} and Corollary \ref{corollary1}.
\end{proof}

This confirms \eqref{eq2} and we turn to \eqref{sum=}. We use the following expression from Young \cite{Y91}.
\beq{Young}{
\sum_{i=1}^n\frac{1}{i}=\log n+\g+\frac{1}{2n}+O(n^{-2}),\qquad \text{where $\g$ is Euler's constant.}
}
Let $m_1=\om^{1/4}m$. Observe first that
\begin{align}
\sum_{i=0}^{m_1}\frac{1}{n-i}\sum_{r=i+1}^{n-m}\frac{1}{r}&\leq O\bfrac{\log n}{n^{1/4}}+\sum_{i=n^{3/4}}^{m_1}\frac{1}{n-i}\sum_{r=i+1}^{n-m}\frac{1}{r}\nonumber\\ &\leq o(1)+\frac{1}{n-m_1}\sum_{i=n^{3/4}}^{m_1}\brac{\log\bfrac{n}{i}+\frac{1}{2(n-m)}+ O(n^{-3/2})}\nonumber\\ 
&\leq o(1)+\frac{2}{n}\log\bfrac{n^{m_1}}{m_1!}\nonumber\\
&\leq o(1)+\frac{2m_1}{n}\log\bfrac{ne}{m_1}\nonumber\\
&=o(1). \label{sum}
\end{align}

Then,
\begin{align}
\sum_{r=1}^{n-m}\sum_{i=0}^{r-1}\frac{1}{r(n-i)}&=\sum_{i=0}^{n-m-1}\frac{1}{n-i}\sum_{r=i+1}^{n-m}\frac{1}{r}, \nonumber\\
&=\sum_{i=m_1}^{n-m-1}\frac{1}{n-i}\sum_{r=i+1}^{n-m}\frac{1}{r}+o(1),\nonumber\\
&=\sum_{i=m_1}^{n-m-1}\frac{1}{n-i}\brac{\log\bfrac{n-m}{i}+\frac{1}{2(n-m)}-\frac{1}{2i}+O(i^{-2})}+o(1), \nonumber\\
&=\sum_{i=m_1}^{n-m-1}\frac{1}{n-i}\log\bfrac{n-m}{i}+o(1),\nonumber\\
&=\sum_{j=m+1}^{n-m_1}\frac{1}{j}\log\bfrac{n-m}{n-j}+o(1),\label{int1}\\
&=\int_{x=m+1}^{n-m_1}\frac{1}{x}\log\bfrac{n-m}{n-x}dx+o(1).\nonumber
\end{align}
We can replace the sum in \eqref{int1} by an integral because the sequence of summands is unimodal and the terms are all $o(1)$.

Continuing, we have
\begin{align}
& \int_{x=m+1}^{n-m_1}\frac{1}{x}\log\bfrac{n-m}{n-x}dx\nonumber\\
&=-\int_{x=m+1}^{n-m_1}\frac{1}{x} \log\brac{1-\frac{x-m}{n-m}}dx\nonumber\\
&=\sum_{k=1}^{\infty}\int_{x=m+1}^{n-m_1}\frac{1}{x}\frac{(x-m)^k}{k(n-m)^k}dx\nonumber\\
&=\int_{y=1}^{n-m-m_1}\frac{1}{y+m}\frac{y^k}{k(n-m)^k}dy.\label{I0}
\end{align}
Observe next that 
$$\int_{y=1}^{n-m-m_1}\frac{1}{y+m}\frac{y^k}{k(n-m)^k}dy\leq \int_{y=1}^{n-m-m_1}\frac{y^{k-1}}{k(n-m)^k}dy \leq \frac{1}{k^2}.$$
So,
\beq{I1}{
0\leq \sum_{k=\log n}^{\infty}\int_{x=m+1}^{n-m_1}\frac{1}{x}\frac{(x-m)^k}{k(n-m)^k}dx\leq \sum_{k=\log n}^\infty \frac{1}{k^2}=o(1).
}
If $1\leq k\leq \log n$ then we write
$$\int_{y=1}^{n-m-m_1}\frac{1}{y+m}\frac{y^k}{k(n-m)^k}dy= \int_{y=1}^{n-m-m_1}\frac{(y+m)^{k-1}}{k(n-m)^k}dy
+\int_{y=1}^{n-m-m_1}\frac{y^k-(y+m)^{k}}{(y+m)k(n-m)^k}dy.$$
Now
\beq{I2}{
\int_{y=1}^{n-m-m_1}\frac{(y+m)^{k-1}}{k(n-m)^k}dy =\frac{1}{k^2}\frac{(n-m_1)^{k}-(m+1)^k}{(n-m)^k}= \frac{1}{k^2}+O\bfrac{1}{k\om^{1/4}\log n}.
}
If $k=1$ then
$$\int_{y=1}^{n-m-m_1}\frac{(y+m)^{k}-y^k}{(y+m)k(n-m)^k}dy\leq\frac{m\log(n-m_1)}{n-m}=o(1).$$
And if $2\leq k\leq \log n$ then
\begin{align*}
\int_{y=1}^{n-m-m_1}\frac{(y+m)^{k}-y^k}{(y+m)k(n-m)^k}dy& =\sum_{l=1}^k\int_{y=1}^{n-m-m_1}\binom{k}{l} \frac{y^{k-l}m^l}{(y+m)k(n-m)^k}dy\nonumber\\
&\leq \sum_{l=1}^k\int_{y=0}^{n-m-m_1}\binom{k}{l} \frac{y^{k-l-1}m^l}{k(n-m)^k}dy\nonumber\\
&= \sum_{l=1}^k\binom{k}{l}\frac{m^l(n-m-m_1)^{k-l}}{k(k-l)(n-m)^k}\nonumber\\
&=O\bfrac{km}{k(k-1)n}=O\bfrac{1}{k\om^{1/2}\log n}
\end{align*}
It follows that
\beq{I3}{
0\leq \sum_{k=1}^{\log n}\int_{y=1}^{n-m-m_1}\frac{(y+m)^{k}-y^k}{(y+m)k(n-m)^k}dy = o(1)+O\brac{\sum_{k=2}^{\log n}\frac{1}{k\om^{1/2}\log n}}=o(1).
}
Equation \eqref{sum=} now follows from \eqref{I0}, \eqref{I1}, \eqref{I2} and \eqref{I3}.

Turning to \eqref{Cnm=} we prove the following lemma:
\begin{lemma}\label{lem3}
If $r\geq n-m$ then $0\leq C(n,r+1)-C(n,r)=O\bfrac{\log n}{np}$.
\end{lemma}
This will prove that
$$0\leq \E{C(n,n)-C(n-m+1)}=O\bfrac{m\log n}{np}=O\bfrac{n}{\om^{1/2}np}=o\bfrac{1}{p}$$
and complete the proof of \eqref{Cnm=} and hence Theorem \ref{th1}.
\subsection{Proof of Lemma \ref{lem3}}\label{pol3}
Let $w(e)$ denote the weight of edge $e$ in $G$. Let $V_r=A_{r+1}\cup B$ and let $G_r$ be the subgraph of $G$ induced by $V_r$. For a vertex $v\in V_r$ order the neighbors $u_1,u_2,\ldots,$ of $v$ in $G_r$ so that $w(v,u_i)\leq w(v,u_{i+1})$. Define the $k$-neighborhood $N_k(v)=\set{u_1,u_2,\ldots,u_k}$. 

Let the $k$-neighborhood of a set be the union of the $k$-neighborhoods of its vertices. In particular, for $S \subseteq A_{r+1}$, $T \subseteq B$,
\begin{align} 
N_k(S) & =\{b\in B:\;\exists a\in S:y\in N_k(a)\},
\label{NS}
\\
N_k(T) & = \{a\in A_{r+1}:\;\exists b\in T:a\in N_k(b)\}.
\label{NT}
\end{align}

Given a function $\f$ defining a matching $M$ of $A_r$ into $B$, we define the following digraph: let $\vG_r=(V_r,\vX)$ where $\vX$ is an orientation of 
\begin{multline*}
X=\\
\set{\set{a,b}\in G: a\in A_{r+1},b\in N_{40}(a)}\cup \set{\set{a,b}\in G: b\in B,a\in N_{40}(b)}\cup\set{(\f(a_i),a_i):i=1,2,\ldots,r}.
\end{multline*}
An edge $e\in M$ is oriented from $B$ to $A$ and has weight $w_r(e)=-w(e)$. The remaining edges are oriented from $A$ to $B$ and have weight equal to their weight in $G$.

The arcs of directed paths in $\vG_r$ are alternately forwards $A\to B$ and backwards $B\to A$ and so they correspond to alternating paths with respect to the matching $M$. It helps to know (Lemma \ref{cl1}, next) that given $a\in A_{r+1},b\in B$ we can find an alternating path from $a$ to $b$ with $O(\log n)$ edges. The $ab$-diameter will be the maximum over $a\in A_{r+1},b\in B$ of the length of a shortest path from $a$ to $b$.
\begin{lemma}\label{cl1}
W.h.p., for every $\f$, the (unweighted) $ab$-diameter of $\vG_r$ is at most $k_0=\rdup{3\log_4n}$.
\end{lemma}
\begin{proof}
For $S\subseteq A_{r+1}$, $T \subseteq B$, let 
\begin{align*}
N(S) &=\{b\in B:\;\exists a\in S\text{ such that } (a,b)\in \vX\},\\
N(T) &=\{a\in A_{r+1}:\;\exists b\in T\text{ such that } (a,b)\in \vX\} . 
\end{align*}
We first prove an expansion property: that \whp,
for all $S\subseteq A_{r+1}$ with $|S|\leq \rdup{n/5}$,
$|N(S)|\geq 4|S|$. 
(Note that $N(S),N(T)$ involve edges oriented from $A$ to $B$ and so do not depend on $\f$.)
\begin{align}\label{ohoh}
\Pr(\exists S: \; |S|\leq \rdup{n/5}, \, |N(S)| < 4|S|)
&\leq o(1)+\sum_{s=1}^{\rdup{n/5}}\binom{r+1}{s}\binom{n}{4s}
\left(\frac{\binom{4s}{40}}{\binom{n}{40}}\right)^s  \notag \\
&\leq\sum_{s=1}^{\rdup{n/5}}\left(\frac{ne}{s}\right)^s
\left(\frac{ne}{4s}\right)^{4s}\left(\frac{4s}{n}\right)^{40s}\notag\\
&=\sum_{s=1}^{\rdup{n/5}}\left(\frac{e^54^{36}s^{35}}{n^{35}}\right)^{s}\notag\\
&=o(1).
\end{align}
{\bf Explanation:} 
{\em The $o(1)$ term accounts for the probability that each vertex has at least 40 neighbors in $\vG_r$. Condition on this. Over all possible ways of choosing $s$ vertices and $4s$ ``targets'', we take the probability that for each of the $s$ vertices, all 40 out-edges fall among the $4s$ out of  the $n$ possibilities.}

Similarly, w.h.p., for all $T\subseteq B$ with $|T|\leq \rdup{n/5}$, $|N(T)|\geq 4|T|$. Thus by the union bound, w.h.p. both these events hold. In the remainder of this proof we assume that we are in this ``good'' case, in which all small sets $S$ and $T$ have large vertex expansion.

Now, choose an arbitrary $a \in A_{r+1}$, and define $S_0,S_1,S_2,\ldots$ as the endpoints of all alternating paths starting from $a$ and of lengths $0,2,4,\dots$.  That is,
$$S_0=\{a\}\mbox{ and }S_i=\f^{-1}(N(S_{i-1})).$$ 
Since we are in the good case, $|S_i|\geq 4|S_{i-1}|$ provided $|S_{i-1}|\leq n/5$, and so there exists a smallest index $i_S$ such that $|S_{i_S-1}| > n/5$, and $i_S-1 \leq \log_4 (n/5) \leq \log_4 n -1$.
Arbitrarily discard vertices from $S_{i_S-1}$ to create a smaller set $S'_{i_S-1}$ with $|S'_{i_S-1}| = \rdup{n/5}$, so that $S'_{i_S} = N(S'_{i_S-1})$ has cardinality $|S'_{i_S}| \geq 4 |S'_{i_S-1}| \geq 4n/5$.

Similarly, for an arbitrary $b \in B$, define $T_0,T_1,\ldots,$ by 
$$T_0=\{b\}\text{ and }T_i=\f(N(T_{i-1})).$$ 
Again, we will find an index $i_T \leq \log_4 n$ whose modified set has cardinality $|T'_{i_T}| \geq 4n/5$.

With both $|S'_{i_S}|$ and $|T'_{i_T}|$ larger than $n/2$, there must be some $a' \in S'_{i_S}$ for which $b'=\f(a') \in T'_{i_T}$. This establishes the existence of an alternating walk and hence (removing any cycles) an alternating path of length at most $2(i_S+i_T) \leq 2 \log_4 n$ from $a$ to $b$ in $\vG_r$.
\end{proof}

We will need the following lemma, 
\begin{lemma}\label{lemFG}
Suppose that $k_1+k_2+\cdots+k_M\leq a\log N$, and $X_1,X_2,\ldots,X_M$ are independent random variables with $Y_i$ distributed as the $k_i$th minimum of $N$ independent exponential rate one random variables. If $\m>1$ then
$$\Pr\left(X_1+\cdots+X_M\geq \frac{\m a \log N}{N-a\log N}\right)\leq N^{a(1+\log \m-\m)}.$$
\end{lemma}
\begin{proof}
Let $Y_{(k)}$ denote the $k$th smallest of $Y_1,Y_2,\ldots,Y_N$, where we assume that $k=O(\log N)$. Then the density function $f_k(x)$ of $Y_{(k)}$ is
$$f_k(x)=\binom{N}{k}k(1-e^{-x})^{k-1}e^{-x(N-k+1)}$$
and hence the $i$th moment of $Y_{(k)}$ is given by
\begin{align*}
\E{Y_{(k)}^i}&=\int_0^\infty \binom{N}{k}kx^i(1-e^{-x})^{k-1}e^{-x(N-k+1)}dx\\
&\leq \int_0^\infty \binom{N}{k}kx^{i+k-1}e^{-x(N-k+1)}dx\\
&=\binom{N}{k}k\frac{(i+k-1)!}{(N-k+1)^{i+k}}\\
&\leq \brac{1+O\bfrac{k^2}{N}}\frac{k(k+1)\cdots(i+k-1)}{(N-k+1)^i}.
\end{align*}
Thus, if $0\leq t<N-k+1$,
$$\E{e^{tY_{(k)}}}\leq \brac{1+O\bfrac{k^2}{N}}\sum_{i=0}^\infty \bfrac{-t}{N-k+1}^i\binom{-k}{i}=\brac{1+O\bfrac{k^2}{N}}\brac{1-\frac{t}{N-k+1}}^{-k}.$$
If $Z=X_1+X_2+\cdots+X_M$ then if $0\leq t<N-a\log N$,
$$\E{e^{tZ}}=\prod_{i=1}^M\E{e^{tX_i}}\leq \brac{1-\frac{t}{N-a\log N}}^{-a\log N}.$$
It follows that
$$\Pr\left(Z\geq \frac{\m a \log N}{N-a\log N}\right)\leq \brac{1-\frac{t}{N-a\log N}}^{-a\log N} \exp\set{-\frac{t\m a\log N}{N-a\log N}}.$$
We put $t=(N-a\log N)(1-1/\m)$ to minimise the above expression, giving
$$\Pr\left(Z\geq \frac{\m a \log N}{N-a\log N}\right)\leq (\m e^{1-\m})^{a\log N}.$$
\end{proof}
\begin{lemma}\label{shortpaths}
W.h.p., for all $\f$, the weighted $ab$-diameter of $\vG_r$ is at most $c_1 \frac{\log n}{np}$ for some absolute contant $c_1>0$.
\end{lemma}
\begin{proof}
Let 
\begin{equation}\label{lexp}
Z_1=\max\left\{\sum_{i=0}^{k}w(x_i,y_i)-\sum_{i=0}^{k-1}w(y_i,x_{i+1})\right\},
\end{equation}
where the maximum is over sequences $x_0,y_0,x_1,\ldots,x_k,y_k$ where $(x_i,y_i)$ is one of the 40 shortest arcs
leaving $x_i$ for $i=0,1,\ldots,k\leq k_0=\rdup{3\log_4n}$,
and $(y_i,x_{i+1})$ is a backwards matching edge.

We compute an upper bound on the probability that $Z_1$ is large.
For any $\z>0$ we have
\begin{multline*}
\Pr\left(Z_1\geq \z\frac{\log n}{np}\right)\leq o(n^{-4})+\sum_{k=0}^{k_0}((r+1)n)^{k+1}\bfrac{1+o(1)}{n}^{k+1}\times p^{k-1}\times\\
\int_{y=0}^\infty \left[\frac{1}{(k-1)!}\lfrac{y\log n}{np}^{k-1}\sum_{\r_0+\r_1+\cdots+\r_{k}\leq 40(k+1)} q(\r_0,\r_1,\ldots,\r_{k};\z+y)\right]dy
\end{multline*}
where 
$$q(\r_0,\r_1,\ldots,\r_{k};\eta)  = \Pr\left(X_0+X_1+\cdots+X_{k}\geq \eta\frac{\log n}{np}\right),$$
$X_0,X_1,\ldots,X_{k}$ are independent and $X_j$ is distributed as the $\r_j$th minimum of $r$ independent exponential random variables. (When $k=0$ there is no term $\frac{1}{(k-1)!}\lfrac{y\log n}{n}^{k-1}$).

{\bf Explanation:} 
{\em The $o(n^{-4})$ term is for the probability that there is a vertex in $V_r$ that has fewer than $(1-o(1))np$ neighbors in $V_r$. We have at most $((r+1)n)^{k+1}$ choices for the sequence $x_0,y_0,x_1,\ldots,x_k,y_k$. The term $\frac{1}{(k-1)!}\lfrac{y\log n}{np}^{k-1}dy$ bounds the probability that the sum of $k$ independent exponentials, $w(y_0,x_1)+\cdots+w(y_{k-1},x_{k})$, is in $\frac{\log n}{np}[y,y+dy]$. (The density function for the sum of $k$ independent exponentials is $\frac{x^{k-1}e^{-x}}{(k-1)!}$.) We integrate over $y$.\\ 
$\frac{(1+o(1))p}{np}$ is the probability that $(x_i,y_i)$ is and edge of $G$ and is the $\r_i$th shortest edge leaving $x_i$, and these events are independent for $0\leq i\leq k$. The factor $p^{k-1}$ is the probability that the $B$ to $A$ edges of the path exist. The final summation bounds the probability that the associated edge lengths sum to at least $\frac{(\z+y)\log n}{np}$.}

It follows from Lemma \ref{lemFG} with $a\leq 3,N=(1+o(1))np,\m=(\z+y)/a$ that if $\z$ is sufficiently large then, for all $y\geq 0$,
$$q(\r_1,\ldots,\r_k;\z+y)  \leq (np)^{-(\z+y)\log n/(2\log np)}=n^{-(\z+y)/2}.$$
Since the number of choices for $\r_0,\r_1,\ldots,\r_k$ is at most $\binom{41k+40}{k+1}$ (the number of positive integral solutions to $a_0+a_1+\ldots+a_{k+1}\leq 40(k+1)$) we have
\begin{align*}
\Pr\left(Z_1\geq \z\frac{\log n}{np}\right)&\leq o(n^{-4})+2n^{2-\z/2}\sum_{k=0}^{k_0}\frac{(\log n)^{k-1}}{(k-1)!}\binom{41k+40}{k+1}\int_{y=0}^\infty y^{k-1}n^{-y/2}dy\\ 
&\leq o(n^{-4})+2n^{2-\z/2}\sum_{k=0}^{k_0}\frac{(\log n)^{k-1}}{(k-1)!}2^{41k+40}\bfrac{2}{\log n}^{k-2}\int_{z=0}^\infty z^{k-1}e^{-z}dz\\ 
&= o(n^{-4})+2^{41}n^{2-\z/2}\log n\sum_{k=0}^{k_0}2^{41k}\\ 
&=o(n^{-4}),
\end{align*}
for $\z$ sufficiently large. 
\end{proof}
Lemma \ref{shortpaths} shows that with probability $1-o(n^{-4})$ in going from $M_r$ to $M_{r+1}$ we can find an augmenting path of weight at most $\frac{c_1\log n}{np}$. This completes the proof of Lemma \ref{lem3} and Theorem \ref{th1}. (Note that to go from w.h.p. to expectation we use the fact that w.h.p. $w(e)=O(\log n)$ for all $e\in A\times B$,)  
\proofend

Notice also that in the proof of Lemmas \ref{cl1} and \ref{shortpaths} we can certainly make the failure probability less than $n^{-K}$ for any constant $K>0$. 
\section{Proof of Theorem \ref{th2}}

Just as the proof method for $K_{n, n}$ in \cite{W1}, \cite{W2} can be modified to apply to $G_{n, n, p}$, the proof for $K_n$ in \cite{W0} can be modified to apply to $G_{n, p}$. 
\subsection{Outline of the proof}
This has many similarities with the proof of Theorem \ref{th1}. The differences are subtle. The first is to let $M_r^*$ be the minimum cost matching of size $r$ after ading a special vertex $v_{n+1}$. It is again important (Lemma \ref{lem10}) to estimate the probability that $v_{n+1}\in M_r^*$. The approach is similar to that for Theorem \ref{th1}, except that we now need to prove separate lower and upper bounds for this probability $P(n,r)$.

\subsection{Proof details}
Consider $G = G_{n, p}$, and denote the vertex set by $V = \set{v_1, v_2, \dots, v_n}$. Add a special vertex $v_{n+1}$ with $E(\l)$-cost edges to all vertices of $V$, and let $G^*$ be the extended graph on $V^*  = V \cup \{v_{n+1}\}$. Say that $v_1, \dots, v_n$ are {\em ordinary}. Let $M_r^*$ be the minimum cost $r$-matching (one of size $r$) in $G^*$, unique with probability one. (Note the change in definition.)  Define $P(n, r)$ as the normalized probability that $v_{n+1}$ participates in $M_r^*$, i.e.
\begin{equation}
P(n, r) = \lim_{\l \rightarrow 0} \frac{1}{\l} \Prob{v_{n+1} \ \mbox{participates in} \ M_r^*}
\end{equation}

Let $C(n, r)$ denote the cost of the cheapest $r$-assignment of $G$. To estimate $C(n, r)$, we will again need to estimate $P(n, r)$, by the following lemma.
\begin{lemma}\label{lem10}
\begin{equation}
\E{C(n, r) - C(n-1, r-1)} = \frac{1}{n}P(n, r)
\end{equation}
\end{lemma}

\begin{proof}
Let $X = C(n, r)$ and $Y = C(n-1, r-1)$. Fix $i\in [n]$ and let $w$ be the cost of the edge $(v_i, v_{n+1})$, and let $I$ denote the indicator variable for the event that the cost of the cheapest $r$-assignment that contains this edge is smaller than the cost of the cheapest $r$-assignment that does not use $v_{n+1}$. The rest of the proof is identical to the proof of Lemma \ref{lem10a}, except that there are now $n$ choices for $i$ as opposed to $r$ in the previous lemma.
\end{proof}
In this case, unlike the bipartite case, we are unable to directly find an asymptotic expression for $P(n, r)$, as we did in Lemma \ref{lem3} and Corollary \ref{corollary1}. Here we will have to turn to bounding $P(n, r)$ from below and above.

\subsection{A lower bound for $P(n, r)$}\label{lowb}

We will consider an algorithm that finds a set $A_s \subseteq V^*$ which contains the set $B_s$ of vertices participating in $M_r^*$, $s = |A_s| \geq |B_s| = 2r$. Call $A_s$ the set of {\em exposed} vertices.

Initially let $A_s = B_s = \emptyset$ and $r = s = 0$. At stage $s$ of the algorithm we condition on

\hspace{1in}$A_s, B_s$ and the existence and cost of all edges within $A_s$.\\ 
\mbox{\hspace{1in}}In particular, we condition on $r$ and the minimum $r$-matching $M_r^*$.

Given a minimum matching $M_r^*$, we decide how to build a proposed $(r+1)$-matching by comparing the following numbers and picking the smallest. 
\begin{itemize}
\item[(a)] $z_a$ equals the cost of the cheapest edge between a pair of unexposed vertices.
\item[(b)] $z_b=\min\{c_1(v) : v\in A_s \setminus B_s\}$, where $c_1(v)$ is the cost of the cheapest edge between $v$ and a vertex $\t_1(v)\notin A_s$.
\item[(c)] $z_c=\min\{c_1(v) + c_1(u) + \d(u, v) : u,v\in B_s\}$ where $\d(u, v)$ denotes the cost of the cheapest alternating path from $u$ to $v$, with the cost of edges in $M_r^*$ taken as the negative of the actual value.
\end{itemize}
Let 
$$z_{\min}=\min\set{z_a,z_b,z_c}.$$

If $z_{\min}=z_a$ then we reveal the edge $\{v, w\}$ and add it to $M_r^*$ to form $M_{r+1}^*$. Once $v, w$ have been determined, they are added to $A_s$ and $B_s$, and we move to the next stage of the algorithm, updating $s \leftarrow s+2, r \leftarrow r+1$.

If $z_{\min}=z_b$ then let $v \in A_s \setminus B_s$ be the vertex with the cheapest $c_1(v)$. We reveal $w=\t_1(v)$ and add $w$ to $A_s$ and to $B_s$ while adding $v$ to $B_s$. Now $M_{r+1}^* = M_r^* \cup \{v, w\}$. We move to the next stage of the algorithm, updating $s \leftarrow s+1, r \leftarrow r + 1$.

If $z_{\min}=z_c$ then reveal $w_1=\t_1(u), w_2=\t_1(v)$. If $w_1 = w_2$, we say that we have a {\em collision}. In this case, the vertex $w_1$ is added to $A_s$ (but not $B_s$), and we move to the next stage with $s \leftarrow s+1$. If there is no collision, we update $M_r^*$ by the augmenting path $w_1, u, \dots, v, w_2$ to form $M_{r+1}^*$. We add $w_1, w_2$ to $A_s$ and $B_s$, and move on to the next stage with $s \leftarrow s + 2$ and $r \leftarrow r+1$.

It follows that $A_s\setminus B_s$ consists of unmatched vertices that have been the subject of a collision.

It will be helpful to define $A_s$ for all $s$, so in the cases where two vertices are added to $A_s$, we add them sequentially with a coin toss to decide the order.

The possibility of a collision is the reason that not all vertices of $A_s$ participate in $M_r^*$. However, the probability of a collision at $v_{n+1}$ is $O(\l^2)$, and as $\l \rightarrow 0$ this is negligible. In other words, as $\l \to 0$,
\begin{equation}
\Prob{v_{n+1} \in M_r^*} \geq \Prob{v_{n+1} \in B_{2r}} = \Prob{v_{n+1} \in A_{2r}} - O(\l^2)
\end{equation}
and we will bound $\Prob{v_{n+1} \in A_{2r}}$ from below.

\begin{lemma}
Conditioning on $v_{n+1} \notin A_s$, $A_s$ is a random $s$-subset of $V$.
\end{lemma}
\begin{proof}
Trivial for $s = 0$. Suppose $A_{s-1}$ is a random $(s-1)$-subset of $V$. Define $N_s(v) = \{w \notin A_s : (v, w) \in E\}$. In stage $s$, if we condition on $d_s(v) = |N_s(v)|$, then under this conditioning $N_s(v)$ is a random $d_s(v)$-subset of $V \setminus A_s$. This is because the constructon of $A_s$ does not require the edges from $A_s$ to $V\setminus A_s$ to be exposed. So, if $A_s \setminus A_{s-1} = \{w\}$ where $w$ is added due to being the cheapest unexposed neighbor of an exposed $v$, then $w$ is a random element of $N_s(v)$ and hence a random element of $V \setminus A_s$.

If we are in case (a), i.e. $M_{r+1}^*$ is formed by adding an edge between two ordinary unexposed vertices $v, w$, then since we only condition on the size of the set $\{(v, w) : v, w \notin A_s\}$, all pairs $v,w \in V \setminus A_s$ are equally likely, and after a coin toss this can be seen as adding two random elements sequentially. We conclude that $A_s$ is a random $s$-subset of $V$.
\end{proof}

Recall that $m = n / (\om^{1/2} \log n)$.
\begin{corollary}\label{cor2}
W.h.p., for all $0 \leq s \leq n-m$ and all $v \in V$,
\begin{equation}
|d_s(v) - (n-s)p| \leq \om^{-1/5}(n-s)p.
\end{equation}
\end{corollary}
\begin{proof}
This follows from the Chernoff bounds as in Lemma \ref{lem2}.
\end{proof}
We now bound the probability that $A_s \setminus A_{s-1} = \{v_{n+1}\}$ from below. There are a few different ways this may happen.

We now have to address some cost conditioning issues. Suppose that we have just completed an iteration. First consider the edges between vertices not in $A_s$. For such an edge $e$, all we know is $w(e)\geq \z$ where $\z=z_{\min}$ of the just completed iteration. So the conditional cost of such an edge can be expressed as $\z+E(1)$ or $\z+E(\l)$ in the case where $e$ is incident with $v_{n+1}$. The exponentials are independent. We only need to compare the exponential parts of each edge cost here to decide the probability that an edge incident with $v_{n+1}$ is chosen. 

We can now consider case (a). Suppose that an edge $\{u, v\}$ between unexposed vertices is added to $A_{s-1}$. By Corollary \ref{cor2}, there are at most $p\binom{n-s+1}{2}(1+\om^{-1/5})$ ordinary such edges. There are $n-s$ edges between $v_{n+1}$ and $V \setminus A_s$, each at rate $\l$. As $\l \rightarrow 0$, the probability that one of the endpoints of the edge chosen in case (a) is $v_{n+1}$ is therefore at least
$$
\frac{\l(n-s)}{\l(n-s) + p\binom{n-s+1}{2}(1+\om^{-1/5})} \geq \frac{1}{p} \frac{2\l}{n-s}(1 - \om^{-1/5}) + O(\l^2)
$$
We toss a fair coin to decide which vertex in the edge $\{u, v\}$ goes in $A_s$. Hence the probability that $A_s \setminus A_{s-1} = \{v_{n+1}\}$ in case (a) is at least
$$
\frac{1}{p}\frac{\l}{n-s}(1 - \om^{-1/5}) + O(\l^2).
$$
We may also have $A_s \setminus A_{s-1} = \{v_{n+1}\}$ if case (a) occurs at stage $s-2$ and $v_{n+1}$ loses the coin toss, in which case the probability is at least
$$
\frac{1}{p} \frac{\l}{n-s+1} (1 - \om^{-1/5})+ O(\l^2). 
$$

Now consider case (b). Here only one vertex is added to $A_{s-1}$, the cheapest unexposed neighbor $w$ of some $v \in A_{s-1} \setminus B_{s-1}$. The cost conditioning here is the same as for case (a), i.e. that the cost of an edge is $\z+E(1)$ or $\z+E(\l)$. By Corollary \ref{cor2}, this $v$ has at most $p(n-s+1)(1+\om^{-1/5})$ ordinary unexposed neighbors, so the probability that $w = v_{n+1}$ is at least
$$
\frac{\l}{p(n-s+1)(1+\om^{-1/5}) + \l} = \frac{1}{p} \frac{\l}{n-s+1}(1 - \om^{-1/5}) + O(\l^2).
$$

Finally, consider case (c). To handle the cost conditioning, we condition on the values $c_1(v)$ for $v\in B_s$. By well-known properties of independent exponential variables, the minimum is located with probability proportional to the rates of the corresponding exponential variables. A collision at $v_{n+1}$ has probability $O(\l^2)$, so assume we are in the case of two distinct unexposed vertices $w_1, w_2$. Suppose that $w_1$ is revealed first. Exactly as in (b), the probability that $w_1 = v_{n+1}$ is at least
$$
\frac{1}{p}\frac{\l}{n-s + 1} (1 - \om^{-1/5})+ O(\l^2).
$$
If $w_1 \neq v_{n+1}$, the probability that $w_2 = v_{n+1}$ (i.e. $A_{s+1} \setminus A_s = \{v_{n+1}\}$) is at least
$$
\frac{1}{p}\frac{\l}{n-s}(1 - \om^{-1/5}) + O(\l^2),
$$
so by considering the possibility that $v_{n+1}$ is the second vertex added from $A_{s-2}$, we again have probability at least
$$
\frac{1}{p}\frac{\l}{n-s+1}(1 - \om^{-1/5}) + O(\l^2).
$$

We conclude that no matter which case occurs, the probability is at least 
$$\frac{1}{p}\frac{\l}{n-s+1}(1 - \om^{-1/5}) + O(\l^2).$$
So
\begin{equation}\label{pklowbound}
P(n, r) \geq \lim_{\l\to 0} \frac{1}{\l} \sum_{s = 1}^{2r} \Prob{A_s \setminus A_{s-1} = \{v_{n+1}\}}\geq \frac{1-\om^{-1/5}}{p} \sum_{s=1}^{2r} \frac{1}{n-s+1}.
\end{equation}
Write 
$$L(n, r) = \sum_{s=1}^{2r} \frac{1}{n-s+1}.$$

\subsection{An upper bound for $P(n, r)$}

We now alter the algorithm above in such a way that $A_{2r}=B_{2r}$. We do not consider $A_s$ for odd $s$ here. At a stage with $s = 2r$, we condition on

\hspace{1in}$A_s$, and the appearance and cost of all edges within $A_s$.\\
\mbox{\hspace{1in}}In particular, we condition on $r$ and the minimum $r$-matching $M_r^*$.
\hspace{1in}A set $C_s \subseteq A_s$, where each $v\in C_s$ has been involved in a collision.

This changes how we calculate a candidate for $M_{r+1}^*$. We now take the minimum of
\begin{itemize}
\item[(a)] $z_a$ equals the cost of the cheapest edge between unexposed vertices.
\item[(b)] $z_b=\min\{c_1(u) + c_1(v) + \d(u, v) : u,v \in A_s, |\{u, v\} \cap C_s| \leq 1\}$, where $c_1$ and $\d$ are as defined in Section \ref{lowb}.
\item[(c)] $z_c=\min\{c_1(u) + c_2(v) + \d(u, v) : u,v \in C_s, \t_1(u) = \t_1(v)\}$, where $\t_1$ is defined in Section \ref{lowb} and $c_2(v)$ is the cost of the second cheapest edge between $v$ and a vertex $\t_2(v)\notin A_s$.
\end{itemize}
Let 
$$z_{\min}=\min\set{z_a,z_b,z_c}.$$

If $z_{\min}=z_a$ then we reveal the edge $\{v, w\}$ and add it to $M_r^*$ to form $M_{r+1}^*$. Once $v, w$ have been determined, they are added to $A_s$ and we move to the next stage of the algorithm, updating $s \leftarrow s+2$.

If $z_{\min}=z_b$ then reveal $w_1=\t_1(u), w_2=\t_1(v)$. If $w_1 = w_2$ then we add $u,v$ to $C_s$ and go to the next stage of the algorithm without changing $s$. (The probability that $\t_1(u) = \t_1(v) = v_{n+1}$ is $O(\l^2)$, and we can safely ignore this as $\l \to 0$). If at some later stage $w_1$ is added to $A_s$ and $u$ say is still in $A_s$ then we remove $u$ from $C_s$. If $w_1\neq w_2$ then we update $M_r^*$ by the augmenting path $w_1, u, \dots, v, w_2$ to form $M_{r+1}^*$. We add $w_1, w_2$ to $A_s$, and move on to the next stage with $s \leftarrow s + 2$. 

If $z_{\min}=z_c$ then we update $M_r^*$ by the augmenting path $w_1=\t_1(u), u, \dots, v, w_2=\t_2(v)$ to form $M_{r+1}^*$. We add $w_1, w_2$ to $A_s$, and move on to the next stage with $s \leftarrow s + 2$.

Eventually we will construct $M_{r+1}^*$ since case (b) with $\t_1(u) = \t_1(v)$ can happen at most $s$ times before $C_s = A_s$.

The cost conditioning is the same as we that for computing the lower bound in Section \ref{lowb}, except for the need to deal with $c_2(v),v\in C_s$. For this we condition on $c_2(v)$ and argue that the probability $\d_2(v)=x$ is proportional to the exponential rate for the edge $(v,x)$. At this point we know that $\d_1(v)\neq v_{n+1}$, since we are assuming $\l$ is so small that this possibility can be ignored. So, in this case, we can only add $v_{n+1}$ as $\d_2(v)$ for some $v\in C_s$.

To analyze this algorithm we again need to show that $A_{2r}$ is a uniformly random subset of $V$.
\begin{lemma}
Conditioning on $v_{n+1} \notin A_{2r}$, $A_{2r}$ is a random $2r$-subset of $V$.
\end{lemma}

\begin{proof}
Let $L$ denote the $n\times n$ matrix of edge costs, where $L(i,j)=w(v_i,v_j)$ and $L(i,j)=\infty$ if edge $(v_i,v_j)$ does not exist in $G$. For a permutation $\p$ of $V$ let $L_\p$ be defined by $L_\p(i,j)=L(\p(i),\p(j))$. Let $X,Y$ be two distinct $2r$-subsets of $V$ and let $\p$ be any permutation of $V$ that takes $X$ into $Y$. Then we have
$$\Pr(A_{2r}(L)=X)=\Pr(A_{2r}(L_\p)=\p(X))=\Pr(A_{2r}(L_\p)=Y)=\Pr(A_{2r}(L)=Y),$$
where the last equality follows from the fact that $L$ and $L_\p$ have the same distribution. This shows that $A_{2r}$ is a random $2r$-subset of $V$.
\end{proof}

Let $d_{2r}(v) = |\{w \notin A_{2r} : (v, w) \in E\}|$.
\begin{corollary}\label{COR1}
W.h.p., for all $0 \leq r \leq (n - m) / 2$ and all $v \in V$,
\begin{equation*}
|d_{2r}(v) - p(n-2r)| \leq \om^{-1/5}p(n-2r).
\end{equation*}
\end{corollary}
\begin{proof}
The proof is again via Chernoff bounds, see Lemma \ref{lem2}.
\end{proof}
We bound the probability that $v_{n+1} \in A_{2r} \setminus A_{2r-2}$ from above. Suppose we are at a stage where a collisionless candidate for $M_r^*$ has been found.

In case (a), as in the previous section the probability that $v_{n+1}$ is one of the two unexposed vertices is at most
$$
\frac{\l(n-2r+2)}{\l(n-2r+2) + p\binom{n-2r+2}{2}(1 - \om^{-1/5})} = \frac{1}{p}\frac{2\l}{n-2r + 1}(1+\om^{-1/5}) + O(\l^2)
$$

Now suppose we are in case (b) with $u, v\notin C_{2r-2}$. If no collision occurs, the probability that one of $\t_1(u), \t_1(v)$ is $v_{n+1}$ is at most
$$
\frac{\l}{\l+d_s(u)} + \frac{\l}{\l + d_s(v)-1} \leq \frac{1}{p}\frac{2\l}{n-2r+1}(1 + \om^{-1/5}) + O(\l^2)
$$
Finally, if we find $M_{r+1}^*$ by alternating paths where one exposed vertex uses its second-cheapest edge to an unexposed vertex, the probability of that vertex being $v_{n+1}$ is even smaller at $\l / (n-2r+1)$. So,
\begin{equation}\label{pkupbound}
P(n,r) = \lim_{\l\to 0}\sum_{s = 1}^r \Prob{v_{n+1} \in A_{2s}\setminus A_{2s-2}} \leq \frac{2(1 + \om^{-1/5})}{p} \sum_{s=1}^r \frac{1}{n-2s+1}
\end{equation}
Write 
$$U(n, r) = \sum_{s=1}^r \frac{2}{n-2s+1}.$$

\subsection{Calculating $\E{C(G_{n, p})}$}

From Lemma \ref{lem10} and (\ref{pklowbound}) we have
\begin{align}
&\E{C(n, (n-m)/2)}\nonumber\\
& = \sum_{r=1}^{(n-m)/2} \frac{1}{n-r+1} P(n-r+1, (n-m)/2 - r + 1)  \nonumber \\
&\geq \frac{1 + o(1)}{p}\sum_{r=1}^{(n-m)/2} \frac{1}{n-r+1} L(n-r+1, (n-m)/2 - r + 1)\label{Lhere}\\
&= \frac{1 + o(1)}{p}\sum_{r=1}^{(n-m)/2} \frac{1}{n-r+1} \sum_{s=1}^{n-m-2r+2} \frac{1}{(n-r+1)-s+1}\nonumber \\
&=\frac{1 + o(1)}{p}\sum_{r=1}^{(n-m)/2} \frac{1}{n-r+1} \brac{\log\left(\frac{n-r+1}{m+r}\right) + \frac{1}{2(n-r)} - \frac{1}{2(m+r)} + O(m^{-2})},\label{add1}
\end{align}
by (\ref{Young}).

The correction terms are easily taken care of. First we have
\begin{align*}
& \left|\sum_{r=1}^{(n-m)/2}\frac{1}{n-r+1}\left(\frac{1}{2(n-r)} - \frac{1}{2(m+r)} + O(m^{-2})\right)\right| \\
= & O\left(\frac{1}{m} \sum_{r=1}^{(n-m)/2}\frac{1}{n-r+1} \right)\\
= & O\bfrac{n-m}{mn} \\
=& o(1).
\end{align*}

Now we want to replace the $(m+r)$ term in the logarithm in the RHS of \eqref{add1} by $r$. For this we let $m_1 = n / (\omega^{1/4} \log n) = m\omega^{1/4}$. Then
\begin{align*}
& \left| \sum_{r=1}^{(n-m)/2} \frac{1}{n-r+1} \log\left(\frac{r}{m+r}\right)\right| \\
= & \sum_{r=1}^{m_1 - 1} \frac{1}{n-r+1}\log\left(1 + \frac{m}{r}\right) + \sum_{r = m_1}^{(n-m)/2}\frac{1}{n-r+1}\log\left(1 + \frac{m}{r}\right) \\
\leq & \log m \sum_{r=1}^{m_1 - 1} \frac{1}{n-r+1} + \log\left(1 + \frac{m}{m_1}\right)\sum_{r = m_1}^{(n-m)/2}\frac{1}{n-r+1} \\
\leq & \log n \frac{m_1}{n-m_1} + \frac{m}{m_1} \frac{(n-m)/2}{n/2} \\
= & o(1).
\end{align*}
So,
\begin{align*}
p\times RHS(\eqref{add1}) = & o(1)+\sum_{r=1}^{(n-m)/2} \frac{1}{n-r+1} \log\left(\frac{n-r}{r}\right) \\ 
= & o(1)+\int_0^{1/2} \frac{1}{1-\a}\log\left(\frac{1-\a}{\a}\right)d\a.
\end{align*}
Substituting $y = \log(1/\a - 1)$ we have
\begin{align}
\int_0^{1/2} \frac{1}{1-\a}\log\left(\frac{1-\a}{\a}\right)d\a & = \int_0^\infty \frac{y}{e^y + 1} dy \nonumber\\
& = \int_0^\infty \frac{ye^{-y}}{1 + e^{-y}} dy \nonumber\\
& = \sum_{j=0}^\infty \int_0^\infty ye^{-y}(-e^{-y})^jdy \nonumber\\
& = \sum_{j=1}^\infty (-1)^{j+1} \frac{1}{j^2} \nonumber\\
&=\frac12\sum_{k=1}^\infty\frac{1}{k^2} \nonumber\\
& = \frac{\p^2}{12}. \label{Mlim}
\end{align}
This proves a lower bound for $\E{C(n, (n-m)/2)}$. It also shows that
\beq{L=}{
\sum_{r=1}^{(n-m)/2} \frac{1}{n-r+1} L(n-r+1, (n-m)/2 - r + 1)=(1+o(1))\frac{\p^2}{12}.
}

For an upper bound, note that for $r \leq (n-m)/2$,
\begin{align*}
U(n, r) - L(n, r) & = \sum_{s = 1}^{r} \left(\frac{1}{n-2s+1} - \frac{1}{n-2s + 2}\right) \\
& = \sum_{s=1}^{r} \frac{1}{(n-2s+1)(n-2s+2)} \\
& = O\bfrac{r}{(n-2r)^2}
\end{align*}
So,
\begin{align}
&\E{C(n, (n-m)/2)} \nonumber\\
& \leq \frac{1+o(1)}{p}\sum_{r=1}^{(n-m)/2} \frac{1}{n-r+1} U(n-r+1, (n-m)/2 - r + 1) \nonumber \\
&=\frac{1+o(1)}{p}\sum_{r=1}^{(n-m)/2} \frac{1}{n-r+1} \brac{L(n-r+1, (n-m)/2 - r + 1)+O\bfrac{r}{(n-2r)^2}} \nonumber\\
&=\frac{1+o(1)}{p}\brac{o(1)+\sum_{r=1}^{(n-m)/2} \frac{1}{n-r+1} \brac{L(n-r+1, (n-m)/2 - r + 1)}}\label{to1}\\
& = \frac{\p^2}{12p}(1+o(1)).\label{to2}
\end{align}
To get from \eqref{to1} to \eqref{to2} we use \eqref{L=}.

We show that for $n$ even, $\E{C(n, n/2) - C(n, (n-m)/2)} = o(1/p)$ to conclude that
$$
\E{C(G_{n,p})} = \E{C(n, n/2) - C(n, (n-m)/2)} + \E{C(n, (n-m)/2)} = \frac{\p^2}{12p}(1+o(1)).
$$
As above, this will follow from the following lemma.
\begin{lemma}\label{lem13}
Suppose $n$ is even. If $(n-m)/2 \leq r \leq n/2$ then $0\leq \E{C(n,r+1)-C(n,r)}=O\bfrac{\log n}{np}$.
\end{lemma}

\subsection{Proof of Lemma \ref{lem13}}

This section will replace Section \ref{pol3}. Let $M=\set{{v,\f(v)},v\in [n]},\f^2(v)=v$ for all $v\in [n]$ be an arbitrary perfect matching of $G_{n,p}$. We let $\vec{M}=\set{(u,v):v=\f(u)}$ consist of two oppositely oriented copies of each edge of $M$. We then randomly orient the edges of $G_{n,p}$ that are not in $M$ and then add $\vec{M}$ to obtain the digraph $\vec{G}=\vec{G}_{n,p}$. Because $np=\om(\log n)^2$ we have that w.h.p. the minimum in- or out-degree in $\vec{G}_{n,p}$ is at least $\om(\log n)^2/3$. Let $\cD$ be the event that all in- and out-degrees are at least this large. Let the {\em $M$-alternating diameter} of $\vec{G}$ be the maximum over pairs of vertices $u\neq v$ of the minimum length of an odd length $M$-alternating path w.r.t. $M$ between $u$ and $v$ where (i) the edges are oriented along the path in the direction $u$ to $v$, (ii) the first and last edges are not in $M$. Given this orientation, we define $\vec{\G}_r$ to be the subdigraph of $\vec{G}$ consisting of the $r$ cheapest non-$\vec{M}$ out-edges from each vertex together with $\vec{M}$. Once we can show that the $M$-alternating diameter of $\vec{\G}_{20}$ is at most $\lceil 3\log_3 n\rceil$, the proof follows the proof of Lemma \ref{shortpaths} more or less exactly. 

\begin{lemma}\label{altdiam}
W.h.p., the alternating diameter of $\vG_{20}$ is at most $k_0=\lceil3\log_3 n \rceil$.
\end{lemma}
\begin{proof}
We first consider the relatively simple case where $np\geq n^{1/3}\log n$. Let $N^+(u)$ be the set of out-neighbors of $u$ in $\vec{G}$ and let $N^-(v)$ be the set of in-neighbors of $v$ in $\vec{G}$. If there is an edge of $\vec{M}$ from $N^+(u)$ to $N^-(v)$ then this creates an alternating path of length three. Otherwise, let $N^{++}(u)$ be the other endpoints of the matching edges incident with $N^+(u)$ and define $N^{--}(v)$ analogously. Note that now we have $N^{++}(u)\cap N^{--}(v)=\emptyset$ and given $\cD$, the conditional probability that there is no edge from $N^{++}(u)$ to $N^{--}(v)$ in $\vec{G}$ is at most $(1+o(1))(1-p)^{(np/3)^2}\leq e^{-(\log n)^3/10}=o(n^{-2})$. Thus in this case there will be an alternating path of length five.
 
Now assume that $np< n^{1/3}\log n$. In which case we can prove 
\beq{smallS}{
\text{W.h.p. $|S|\leq n^{7/12}$ implies that $e(S)\leq 6|S|$.}
}
This follows from
\begin{align}
\Pr(\exists S\text{ violating \eqref{smallS}})&\leq \sum_{s=13}^{n^{7/12}}\binom{n}{s}\binom{\binom{s}{2}}{6s}p^{6s}\\
&\leq \sum_{s=13}^{n^{7/12}}\bfrac{ne}{s}^s\bfrac{sep}{12}^{6s}\\
&\leq \sum_{s=13}^{n^{7/12}}\bfrac{e^7s^5(\log n)^6}{12^6n^3}^s\\
&=o(1).
\end{align}

Imitating Lemma \ref{cl1}, we prove an expansion property for $\vG_{20}$:
\begin{align}
\Pr(\exists S: \; |S|\leq n^{2/3}, \, |N_{20}(S)| < 10|S|)
&\leq o(1)+\sum_{s=1}^{n^{2/3}}\binom{n}{s}\binom{n-s}{10s}
\left(\frac{\binom{11s}{20}}{\binom{n}{20}}\right)^s  \notag \\
&\leq\sum_{s=1}^{n^{2/3}}\left(\frac{ne}{s}\right)^s
\left(\frac{ne}{10s}\right)^{10s}\left(\frac{11s}{n}\right)^{20s}\notag\\
&=\sum_{s=1}^{n^{2/3}}\left(\frac{e^{11} 11^{20}s^{9}}{10^{10}n^{9}}\right)^{s}\notag\\
&=o(1).\label{expand}
\end{align}
Fix an arbitrary pair of vertices $a,b$. Define $S_i, i = 0,1,\dots$ to be the set of vertices $v$ such that there exists a directed $M$-alternating path of length $2i$ in $\vG_{20}$ from $a$ to $v$. We let $S_0=\set{a}$ and given $S_i$ we let $S_i'=N^+(S_i)\setminus \set{b}$ and $S_i''=\set{w\neq b:\exists \set{v,w}\in M:v\in S_i'}$. Here $N^+(S)=\set{w\notin S:\exists v\in S,(v,w)\in E(\vec{\G})}$ is the set of out-neighbors of $S$. $N^-(S)$ is similarly defined as the set of in-neighbors. It follows from \eqref{smallS} and \eqref{expand} that w.h.p. $|S_i''|\geq 3|S|$, so long as $|S_i|=o(n^{7/12})$. We therefore let $S_{i+1}$ be a subset of $S_i''$ of size $3|S_i|$. So w.h.p. there exists an $i_a \leq \log_3 n$ such that $|S_{i_a}| \in [n^{13/24},3n^{13/24}]$. 

Repeat the procedure with vertex $b$, letting $T_0 = \{b\}, T_{j+1}' = N^-(T_j)$ etc. By the same argument, there exists an $j_b \leq \log_3 n$ such that $T_{j_b}$ is of size in $[n^{13/24},3n^{13/24}]$. Finally, the probability that there is no $S_{i_a}\to T_{j_b}$ edge is at most $(1-p)^{n^{13/12}}=o(n^{-2})$. This completes the proof of Lemma \ref{altdiam}.
\end{proof}

The remainder of the proof of Lemma \ref{lem13} is now exactly as in Section \ref{pol3}. This concludes the proof of Theorem \ref{th2}.

\section{Proof of Theorem \ref{thtal}}
The proof of Lemma \ref{shortpaths} allows us to claim that for any constant $K>0$, with probability $1-O(n^{-K})$ the maximum length of an edge in the minimum cost perfect matching of $G$ is at most $\m=c_2\frac{\log n}{np}$ for some constant $c_2=c_2(K)>0$. We can now proceed as in Talagrand's proof of concentration for the assignment problem. We let $\hw(e)=\min\set{w(e),\m}$ and let $\hC(G)$ be the assignment cost using $\hw$ in place of $w$. We observe that 
\beq{conc}{
\Pr(\hC(G)\neq C(G))=O(n^{-K})} 
and so it is enough to prove concentration of $\hC(G)$.

For this we use the following result of Talagrand \cite{tal}: consider a family $\cF$ of $N$-tuples $\A=(\a_i)_{i\leq N}$ of non-negative real numbers. Let
$$Z=\min_{\A\in \cF}\sum_{i\leq N}\a_iX_i$$
where $X_1,X_2,\ldots,X_N$ are an independent sequence of random variables taking values in $[0,1]$.

Let $\s=\max_{\A\in\cF}||\A||_2$. Then if $M$ is the median of $Z$ and $u>0$, we have
\beq{taly}{
\Pr(|Z-M|\geq u)\leq 4\exp\set{-\frac{u^2}{4\s^2}}.
}
We apply \eqref{taly} with $N=n^2$ and $X_e=\hw(e)/\m$. For $\cF$ we take the $n!$ $\set{0,1}$ vectors corresponding to perfect matchings and scale them by $\m$. In this way, $\sum_e\a_eX_e$ will be the weight of a perfect matching. In this case we have $\s^2\leq n\m^2$. Applying \eqref{taly} we obtain
\beq{conc1}{
\Pr\brac{|\hC(G)-\hM|\geq \frac{\e}{p}}\leq 4\exp\set{-\frac{\e^2}{4p^2}\cdot\frac{1}{n\m^2}} = \exp\set{-\frac{\e^2n}{(c_2\log n)^2}},}
where $\hM$ is the median of $\hC(G)$. Theorem \ref{th2} follows easily from \eqref{conc} and \eqref{conc1}.
\section{Final remarks}
We have generalised the result of \cite{A01} to the random bipartite graph $G_{n,n,p}$ and the result of \cite{W0} to the random graph $G_{n,p}$. It would be of some interest to extend the result in some way to random regular graphs. In the absence of proving Conjecture \ref{conj1} we could maybe extend the results of \cite{A01}, \cite{W0} to some special class of special graphs e.g. to the hypercube.

\end{document}